\documentclass[12pt,reqno]{amsart}
\usepackage{amssymb}
\usepackage{amsmath}
\usepackage[active]{srcltx}
\usepackage{t1enc}
\usepackage[latin2]{inputenc}
\usepackage{verbatim}
\usepackage{amsmath,amsfonts,amssymb,amsthm}
\usepackage[mathcal]{eucal}
\usepackage{enumerate}
\usepackage[centertags]{amsmath}
\usepackage{graphics}

\newtheorem{theorem}{Theorem}

\newtheorem{lemma}{Lemma}

\newtheorem{definition}{Definition}

\begin{document}
	\author{Vakhtang Tsagareishvili and Giorgi Tutberidze}
	\title[Absolute convergence factors]{Absolute convergence factors of Lipshitz class functions for general Fourier series}
	\address{ Associate Professor V. Tsagareishvili, Department of Mathematics, Faculty of Exact and Natural Sciences, Ivane Javakhishvili Tbilisi State University, Chavchavadze str. 1, Tbilisi 0128, Georgia }
	\email{cagare@ymail.com}
	\address{ G.Tutberidze, The University of Georgia, School of IT, Engineering and Mathematics, 77a Merab Kostava St, Tbilisi 0128, Georgia and Department of Computer Science and Computational Engineering, UiT - The Arctic University of Norway, P.O. Box 385, N-8505, Narvik, Norway.}
	\email{giorgi.tutberidze1991@gmail.com}
	\thanks{The research was supported by Shota Rustaveli National Science Foundation grant PHDF-18-476.}
	\date{}
	\maketitle

	\begin{abstract}
	 The main aim of this paper is to investigate the sequences of positive numbers, for which multiplication with Fourier coefficients of functions $f\in$ Lip1 class provides absolute convergence of Fourier series. 
	 
	 In particular we found special conditions for the functions of orthonormal system, for which the above sequences are absolute convergence factors of Fourier series of functions of Lip1 class. It is established that the resulting conditions are best possible in certain sense.
			
	\end{abstract}
	
	
	\textbf{Key} \textbf{words and phrases: }Fourier coefficients, Fourier series, absolute convergence, orthonormal systems.

	\section{\textbf{PRELIMINARIES}}

	\begin{definition}
		\label{deffinition 1}	Let $f(x)$ be defined on an interval $\left[0,1\right]$ and suppose that $x, x+h \in \left[0,1\right].$ If $\alpha \in\left(0,1\right]$ and $\left\Vert f\left(x+h\right)-f\left(x\right) \right\Vert_C <M\left\vert h \right\vert^\alpha$ ($M$-absolute constant), then we say that  $f\in Lip\alpha.$

	\end{definition}
	$Lip1$ is Banach space and the norm is given by following equality
	\begin{equation*}
	\parallel f \parallel_{Lip1}:=\sup_{x,y \in \left[0,1\right]} \left|\frac{f\left(x\right)-f\left(y\right)}{x-y}\right|+\parallel f \left(x\right) \parallel_C .
	\end{equation*}

	Let $(\varphi _{n})$\ be an orthonormal system on $[0,1]$
	and let 
	\begin{equation*}
	c_{n}(f)=\int_{0}^{1}f(x)\varphi _{n}(x)dx,\text{ \ \ \ \ \ \ \ \ }%
	n=1,2,\ldots
	\end{equation*}%
	be the Fourier coefficients of a function $f\in L_2 \left(0,1\right)$.
	
	\begin{definition} [see \cite{TsagareishviliTutberidze}]
		\label{deffinition 2} Let $\left(\varphi_n \right)$ be ONS. We say that sequence of bounded and positive numbers $\left(d_{k}\right) $ are absolute convergence factor of the set of functional class $E$, if
		\begin{equation*}
		\sum_{k=1}^{\infty }\left\vert c_{k}(f)\right\vert k^{\frac{1}{2}}d_{k}<+\infty ,
		\end{equation*}%
		for all $f\in E.$ \label{00}
	\end{definition}
	
	\begin{definition}
		Let   $\left(\varepsilon_n\right)$ is a sequence of numbers  $\varepsilon_n \in \{-1,0,1\}$ and $n=1,2, \dots $. The class of all sequences $\varepsilon =\left(\varepsilon_n\right)$ will be denoted by $\Delta$.         
	\end{definition}
	
	\begin{theorem}[see \cite{Golubov} and \cite{Zygmund},p.387]
		\label{theorem A} If $\left( \varphi _{n}\right) $ is Trigonometric (see \cite{Zygmund}, chapter 1), Haar (see \cite{Ulianov} ) or Walsh (see \cite{KashinSaakyan}, chapter 4) system and $f \in Lip\alpha ,\alpha \in (0;1],$ then%
		\begin{equation*}
		\sum_{n=1}^{\infty }\left\vert c_{n}\left( f\right) \right\vert n^{\gamma}<+\infty,
		\end{equation*}%
		when $\gamma <\frac{\alpha }{2},$ and 
		\begin{equation*}
		c_{n}\left( f\right) =\int_{0}^{1}f\left( x\right) \varphi _{n}\left(
		x\right) dx.
		\end{equation*}
	\end{theorem}
	
	\begin{theorem}[Bochkarev (see \cite{Bochkariov1} and \cite{Bochkarev2})] 
		\label{theorem B} For every complete
		system $\left(\varphi _{n}\right)$ and for every $\alpha \in (0;1]$ there
		exists functions $f \in Lip\alpha $ such that
		\begin{equation*}
		\sum_{n=1}^{\infty }\left\vert c_{n}\left( f_{\alpha }\right) \right\vert n^{\frac{\alpha }{2}}=\infty ,
		\end{equation*}
		
		when $\alpha =1,$ the following is true:
		\begin{equation*}
		\sum_{n=1}^{\infty }\left\vert c_{n}\left( f_1\right) \right\vert n^{\frac{1}{2}}=\infty.
		\end{equation*}%
	\end{theorem}
	
	\begin{theorem} [see \cite{Tsagareishvili}]  
		\label{theorem C} Let $f \in Lip1$, $\left( \varphi _{n}\right) $ be ONS on $\left[0,1\right]$ and
		\begin{equation*}
		\int_{0}^{1}\varphi _{n}\left( x\right) dx=0,
		\end{equation*}
		where $n=m,m+1, ...$ and $f\in Lip1$. Then 
		\begin{equation*}
		\sum_{n=1}^{\infty }\left\vert c_{n}\left( f\right) \right\vert n^{\gamma}<+\infty, \text{     } \left(0<\gamma<1\right)
		\end{equation*}
		if\ and only if 
		\begin{equation*}
		\frac{1}{n}B_{n}^{\left( \gamma \right) }\left( t\right) <M,\text{ }\forall
		t\in (0;1],
		\end{equation*}
		where 
		\begin{equation*}
		B_{n}^{\left( \gamma \right) }\left( t\right) =\sum_{i=1}^{n-1}\left\vert
		\int_{0}^{\frac{i}{n}}\sum_{k=1}^{n}k^{\gamma }\varphi _{k}\left( x\right) r_{k}\left( t\right)
		dx\right\vert
		\end{equation*}
		and $M$ does not depend on $n$ and $t$, and $r_{k}\left( t\right) $ are Rademacher functions (see \cite{Alekits}, ch.1).
	\end{theorem}
	
		In 1962, A. Olevsky (see \cite{Olevsky}) proved that if $\ f\in L_2\left(0,1\right)$ and $(a_{n})\in l_{2}$ then
	there exist orthonormal system (ONS) $\left(\varphi _{n}\right)$, such that 
	\begin{equation*}
	c_{n}(f)=\int_{0}^{1}f\left( x\right) \varphi _{n}\left( x\right) dx=ba_{n}, \text{     } n=1,2,...,
	\end{equation*}
	and $b$ is an absulute constant which does not depend on $n$.

	\begin{theorem}
		\label{theorem D} For $f\left(x\right)=1, x\in \left[0,1\right] $ there exists ONS $\left(\varphi _{n}\right)$ such that for some $0<\gamma <\frac{1}{2},$
		\begin{equation*}
		\sum_{n=1}^{\infty }\left\vert c_{n}\left( f\right) \right\vert n^{\gamma}=+\infty.
		\end{equation*}
	\end{theorem}
	\begin{proof}
		Proof of theorem \ref {theorem D} is obtained from Olevsky's theorem (see \cite{Olevsky}).
		
		Let $f_{0}(x)=1,$ for all $x\in \lbrack 0,1]$ and $a_n=n^{-\frac{4}{3}}, \gamma=\frac{1}{3}$ then
		
		\begin{equation*}
		\sum_{n=1}^{\infty }a_{n}^{2}<+\infty .
		\end{equation*}
		By using theorem of Olevsky (see \cite{Olevsky}), there exists an orthonormal system $\left(\varphi _{n}\right),$ such that $c_{n}\left( f_{0}\right) =b\cdot n^{-\frac{4}{3}}$. It follows that
		
		\begin{equation*}
		\sum_{n=1}^{\infty } c_{n}^2\left( f_{0}\right)=b^2\cdot \sum_{n=1}^{\infty }n^{-\frac{4}{3}}<+\infty ,
		\end{equation*}
		
		and
		
		\begin{equation*}
		\sum_{n=1}^{\infty } n^{\frac{1}{3}}\left|c_{n}\left( f_{0}\right)\right|=\sum_{n=1}^{\infty }n^{-1}=+\infty.
		\end{equation*}
		
		The main aim of this paper is to find a positive bounded sequence $\left(d_{k}\right) $, when for any $f\in Lip1$, holds
		\begin{equation*}
		\sum_{n=1}^{\infty}n^{\frac{1}{2}}\left|c_n\left(f\right)\right|d_n<+\infty.
		\end{equation*}
		
		Similar problems  for the functions of finite variation with respect to general orthonormal systems were studied in the paper  \cite{Tsagareishvili,TsagareishviliTutberidze}.
		
		If $f,F\in L_2\left(0,1\right)$ and $f$ takes finite values at every point of  $\left[0,1\right]$, then next equality is true (see \cite{GogoladzeTsagareishvili})
		\begin{eqnarray}
		\label{1} \int_{0}^{1}f(x)F(x)dx &=& \sum_{i=1}^{N-1}\left(f\left(\frac{i}{N}\right)-f\left(\frac{i+1}{N}\right)\right)\int_{0}^{\frac{i}{N}}F\left(x\right)dx  \\ &+&\sum_{i=1}^{N}\int_{\frac{i-1}{N}}^{\frac{i}{N}}\left(f\left(x\right)-f\left(\frac{i}{N}\right)\right)F\left(x\right)dx \notag \\
		&+& f\left(1\right)\int_{0}^{1}F\left(x\right)dx . \notag
		\end{eqnarray}
	\end{proof}
	
	\begin{lemma}
		\label{lemma 1} Let $Q_n\in L_2$ and $E_n$ be a set of all $i$ $\left(i=1,2,...,n-1\right)$, for which there exists $t\in\left[\frac{i-1}{n},\frac{i}{n}\right]$ such, that 
		\begin{equation}
		sign\int_{0}^{t}Q_n \left(x\right) dx \ne sign\int_{0}^{\frac{i}{n}}Q_n \left(x\right) dx, \label{2}
		\end{equation}
		then
		\begin{equation}
		\sum_{i\in E_n}\left|\int_{0}^{\frac{i}{n}}Q_n \left(x\right) dx\right|\leq \left(\int_{0}^{1}Q_{n}^{2} \left(x\right) dx\right)^{\frac{1}{2}}. \label{3}
		\end{equation}
		
	\end{lemma}
	\begin{proof}
		From (\ref{2}) we get that, there exist $t_{i}\in\left[\frac{i-1}{n},\frac{i}{n}\right]$ for $\forall i\in E_n$ such, that
		\begin{equation*}
		\int_{0}^{t_{i}}Q_n \left(x\right) dx=0.
		\end{equation*}
		Hence,
		\begin{eqnarray}
		\int_{0}^{\frac{i}{n}}Q_n \left(x\right) dx&=&\int_{0}^{t_{i}}Q_n \left(x\right) dx +\int_{t_{i}}^{\frac{i}{n}}Q_n \left(x\right) dx \notag \\
		&=&\int_{t_i}^{\frac{i}{n}}Q_n \left(x\right)dx.   \notag
		\end{eqnarray}
		From here we have
		\begin{eqnarray}
		\sum_{i\in E_n}\left|\int_{0}^{\frac{i}{n}}\varphi_n \left(x\right) dx\right|&\leq& \sum_{i\in E_n}\left|\int_{t_{i}}^{\frac{i}{n}}Q_n \left(x\right)dx\right|  \notag \\
		&\leq& \int_{0}^{1}\left|Q_n \left(x\right)\right|dx \notag \\
		&\leq& \left(\int_{0}^{1}Q_{n}^{2} \left(x\right) dx\right)^{\frac{1}{2}}. \notag
		\end{eqnarray}
	\end{proof}
	
	Set
	\begin{equation}
	Q_N \left(x,\varepsilon\right):=\sum_{k=1}^{N}d_k k^{\frac{1}{2}}\varphi_k \left(x\right)\varepsilon_k, \label{*}
	\end{equation}
	where $\left(\varphi_n \right)$ is ONS. 
	
	Denote by
	\begin{equation}
	D_N\left(\varepsilon\right):=\frac{1}{N}\sum_{i=1}^{N-1}\left|\int_{0}^{\frac{i}{N}}Q_N \left(x,\varepsilon\right)dx \right|.	\label{***} 
	\end{equation}

	\section{\textbf{The Main Results }}
	
	For given $\left( \varphi _{n}\right)$ and every $\ f \in Lip1$ what should be sequence $\left( d_{k}\right)$ such that the following condition holds true:
	\begin{equation*}
	\sum_{n=1}^{\infty }\left\vert c_{n}\left( f\right) \right\vert n^{\frac{1}{2}}d_{n}<+\infty .
	\end{equation*}%

	Our main results read as:
	
	\begin{theorem}
		\label{theorem5} Let $\left(\varphi _{n}\right) $ be an ONS and  
		\begin{equation*}
		\int_{0}^{1}\varphi _{n}\left(x\right)dx=0, \text{ } \text{ } n=1,2,..., .
		\end{equation*}
		Then, the sequence $\left(d_k\right)$ are absolute convergence factor for the $Lip1$ class if 
		\begin{eqnarray}
		D_N\left(\varepsilon\right)=O\left(1\right) \label{5.1}
		\end{eqnarray}
		for all $\varepsilon\in\left(\varepsilon_n\right)\in \Delta.$
	\end{theorem}
	\begin{proof}[\textbf{Proof of theorem \protect\ref{theorem5}.}]
		Let $f\in Lip1$ and $\varepsilon\in\left(\varepsilon_n\right)\in \Delta$ be selected so that 
		\begin{equation*}
		sign c_k\left(f\right)=\varepsilon_k, k=1,2,3, \dots 
		\end{equation*}
	Then
		\begin{equation*}
		\sum_{k=1}^{N}\left\vert c_{k}(f)\right\vert d_{k} k^{\frac{1}{2}}=\sum_{k=1}^{N}d_{k}k^{\frac{1}{2}}c_k \left(f\right)\varepsilon_k.
		\end{equation*}
		Hence (see (\ref{*}))
		\begin{eqnarray}
		\label{8} \sum_{k=1}^{N}\left\vert c_{k}(f)\right\vert d_{k} k^{\frac{1}{2}} &=& \int_{0}^{1}f\left(x\right)\sum_{k=1}^{N}d_{k} k^{\frac{1}{2}}\varphi_k\left(x\right)\varepsilon_k  \text{ } dx\\ 
		&=&\int_{0}^{1}f\left(x\right)Q_N \left(x,\varepsilon\right)dx. \notag
		\end{eqnarray}	
		By applying (\ref{1}) we can conclude that,
		\begin{eqnarray}
		\int_{0}^{1}f\left(x\right)Q_N \left(x,\varepsilon\right)dx \label{9} 
		&=& \sum_{i=1}^{N-1}\left(f\left(\frac{i}{N}\right)-f\left(\frac{i+1}{N}\right)\right)\int_{0}^{\frac{i}{N}}Q_N\left(x,\varepsilon\right)dx \notag \\ &+&\sum_{i=1}^{N}\int_{\frac{i-1}{N}}^{\frac{i}{N}}\left(f\left(x\right)-f\left(\frac{i}{N}\right)\right)Q_N\left(x,\varepsilon\right)dx  \\
		&=& I_1 +I_2. \notag
		\end{eqnarray}
		According to condition of theorem \ref{theorem5}, we get that (see (\ref{***}))
		\begin{eqnarray}
		\left|I_1\right|&=&\left|\sum_{i=1}^{N-1}\left(f\left(\frac{i}{N}\right)-f\left(\frac{i+1}{N}\right)\right)\int_{0}^{\frac{i}{N}}Q_N\left(x,\varepsilon\right)dx\right| \label{10} \\
		&=& O(1)\frac{1}{N}\sum_{i=1}^{N-1} \left|\int_{0}^{\frac{i}{N}}Q_N\left(x,\varepsilon\right)dx\right| \notag \\
		&=&O\left(1\right)D_N\left(\varepsilon\right)= O\left(1\right). \notag
		\end{eqnarray}
	Using the Cauchy inequality for $I_2$ we have that (see (\ref{*}))
		\begin{eqnarray}
		\text{\ \ \ \ \ \ \ \ \ \ \ \ }\left|I_2\right| &=&  \left|\sum_{i=1}^{N}\int_{\frac{i-1}{N}}^{\frac{i}{N}}\left(f\left(x\right)-f\left(\frac{i}{N}\right)\right)Q_N\left(x,\varepsilon\right)dx \right|      \label{11} \\
		&=&  O\left(\frac{1}{N}\right) \int_{0}^{1}\left|Q_N\left(x,\varepsilon\right)\right|dx \notag \\
		&=& O\left(\frac{1}{N}\right) \left(\int_{0}^{1}\left(\sum_{k=1}^{N}d_k k^{\frac{1}{2}}\varphi_k\left(x\right)\varepsilon_k\right)^2 dx\right)^\frac{1}{2} \notag \\
		&=& O\left(\frac{1}{N}\right) \left(\sum_{k=1}^{N} k d_k^2 \right)^{\frac{1}{2}} = O\left(\frac{1}{N} N\right )\max_{1 \leq k \leq N}d_k=O\left(1\right).    \notag
		\end{eqnarray}
		By combining (\ref{9}), (\ref{10}) and  (\ref{11}) we can write that
		\begin{equation}
		\left|\int_{0}^{1}f\left(x\right)Q_N \left(x,\varepsilon\right)dx\right|<c, \label{12.1}
		\end{equation}
		where $c>0$ does not depend on $N$. \\
		
	From (\ref{8}) and (\ref{12.1}), we get
		\begin{equation*}
		\sum_{k=1}^{\infty}\left|c_k\left(f\right)\right|k^\frac{1}{2}d_k<+\infty.
		\end{equation*}
		Theorem \ref{theorem5} is proved.
	\end{proof}

	\begin{theorem}
		\label{theorem6} Let $\left(\varphi _{n}\right) $ be an orthonormal system and $\left(d_n\right)$ is any given positive bounded sequence. If for some  $\varepsilon^0\in\left(\varepsilon_n^0\right)\in \Delta$ 
		\begin{equation}
		\label{12} \overline{\lim_{N\rightarrow \infty }}D_{N}\left( \varepsilon^0\right) =+\infty,
		\end{equation}
		then $\left(d_n\right)$ is not an absolutely convergence factor for $Lip1$ class.
	\end{theorem}
	\begin{proof}[\textbf{Proof of theorem \protect\ref{theorem6}.}]
		Suppose that for some $\varepsilon^0\in\left(\varepsilon_n^0\right)\in \Delta,$ the following holds
		\begin{equation*}
		\overline{\lim_{N\rightarrow \infty }}D_{N}\left( \varepsilon^0\right) =+\infty.
		\end{equation*}
		Now, we consider next sequences of functions 
		\begin{equation}
		\label{13} f_N\left(x\right)=\int_{0}^{x}sign\int_{0}^{y}Q_N\left(z,\varepsilon^0\right)dzdy.
		\end{equation}
		Set 
		\begin{equation*}
		F_N:=\left\{1,2,...,N-1,\right\}\setminus E_N.
		\end{equation*}
		Then by the definition of $F_N$ and using (\ref{13}) we get that
		\begin{eqnarray}
		\sum_{i\in F_N}\left(f_N\left(\frac{i}{N}\right)-f_N\left(\frac{i+1}{N}\right)\right)\int_{0}^{\frac{i}{N}}Q_N\left(x,\varepsilon^0\right)dx  \notag\\
		=-\frac{1}{N}\sum_{i\in F_N\notag} \left|\int_{0}^{\frac{i}{N}}Q_N\left(x,\varepsilon^0\right)dx\right|.
		\end{eqnarray}
		By applying (see (\ref{13}))
		\begin{equation*}
		\left|f_N \left(\frac{i}{N}\right)-f_N \left(\frac{i+1}{N}\right)\right|\leq\frac{1}{N},
		\end{equation*}
	and using the above received equality we have 
		
		\begin{eqnarray}
		\label{14} &&\left|\sum_{i=1}^{N-1}\left(f_N \left(\frac{i}{N}\right)-f_N \left(\frac{i+1}{N}\right)\right)\int_{0}^{\frac{i}{N}}Q_N \left(x,\varepsilon^0\right)dx\right|\\
		&=&\left| \sum_{i\in F_{N}}\left(F_{N}\left(\frac{i}{N}\right)-f_N \left(\frac{i+1}{N}\right)\right)\int_{0}^{\frac{i}{N}}Q_N\left(x,\varepsilon^0\right)dx \right. \notag \\
		&-&\left.\sum_{i\in E_{N}} \int_{\frac{i-1}{N}}^{\frac{i}{N}}\left(F_{N}\left( x \right)-f_N\left(\frac{1}{N}\right)\right)Q_N\left(x,\varepsilon^0\right)dx\right| \notag \\
		&\geq& \frac{1}{N}\sum_{i\in F_N}\left|\int_{0}^{\frac{i}{N}}Q_N \left(x,\varepsilon^0\right)dx\right|- \frac{1}{N}\sum_{i\in E_N}\left|\int_{0}^{\frac{i}{N}}Q_N \left(x,\varepsilon^0\right)dx\right| \notag \\	
		&=&\frac{1}{N}\sum_{i=1}^{N-1}\left|\int_{0}^{\frac{i}{N}}Q_N \left(x,\varepsilon^0\right)dx\right| - \frac{2}{N}\sum_{i\in E_N}\left|\int_{0}^{\frac{i}{N}}Q_N \left(x,\varepsilon^0\right)dx\right|. \notag	
		\end{eqnarray}

		According to Lemma\ref{lemma 1}
		
		\begin{eqnarray}
		\frac{2}{N}\sum_{i\in F_N}\left|\int_{0}^{\frac{i}{N}}Q_N \left(x,\varepsilon^0\right)dx\right|\leq \frac{2}{N}\left(\int_{0}^{1}Q_N^2 \left(x,\varepsilon^0\right)dx \right)^{\frac{1}{2}}\notag \\
		= O\left(\frac{1}{N}\right)\left(\sum_{k=1}^{N}k d_k^2\right)^{\frac{1}{2}} = O\left(\frac{1}{N} N\right )\max_{1 \leq k \leq N}d_k=O\left(1\right). \notag
		\end{eqnarray}

		Finally, we can conclude that (see (\ref{***}) and (\ref{14}))
		\begin{eqnarray}
		\label{15} \left|\sum_{i=1}^{N-1}\left(f_N \left(\frac{i}{N}\right)-f_N \left(\frac{i+1}{N}\right)\right)\int_{0}^{\frac{i}{N}}Q_N \left(x,\varepsilon^0\right)dx\right|
		\end{eqnarray}
		\begin{eqnarray}
		&\geq& \frac{1}{N}\sum_{i=1}^{N-1}\left|\int_{0}^{\frac{i}{N}}Q_N \left(x,\varepsilon^0\right)dx\right|-O\left(1\right) \notag \\
		&=& D_N\left(\varepsilon^0\right) -O\left(1\right).  \notag
		\end{eqnarray}
		By using (\ref{13}), for $x,y\in \left[\frac{i-1}{N},\frac{i}{N}\right], \text{   where }  i=1,2,..., N$		
		\begin{equation*}
		\left|f_N\left(x\right)-f_N\left(y\right)\right|\leq\frac{1}{N}.	
		\end{equation*}
		
	According to Cauchy inequality, we get 
		\begin{eqnarray}
		\label{16} \left|\sum_{i=1}^{N}\int_{\frac{i-1}{N}}^{\frac{i}{N}}\left(f_N\left(x\right)-f_N\left(\frac{i}{N}\right)\right)Q_N\left(x,\varepsilon^0\right)dx\right|
		\end{eqnarray}
		\begin{eqnarray}
	\text{\ \ \ \ \ \ \ \  \ \ \ \ \ \ \ \ \ \ \ \ \ }	&\leq& \frac{1}{N}\int_{0}^{1}\left\vert Q_N\left(x,\varepsilon^0\right) \right\vert dx  \notag \leq \frac{1}{N} \left(\int_{0}^{1}Q_N ^2 \left(x,\varepsilon^0\right)dx\right)^{\frac{1}{2}}  \notag \\
		&=&\frac{1}{N}\left(\sum_{k=1}^{N}k d_k^2\right)^{\frac{1}{2}}
		=O\left(\frac{1}{N} N\right )\max_{1 \leq k \leq N}d_k=O\left(1\right).   \notag
		\end{eqnarray}
		In equality (\ref{9}) suppose that,  $f\left(x\right)=f_N \left(x\right).$ If we consider inequalities (\ref{15}) and (\ref{16}), we obtain
		\begin{eqnarray} \label{20}
	&&\left|\int_{0}^{1}f_N\left(x\right) Q_N\left(x,\varepsilon^0\right)dx\right| \\
	&\geq&\frac{1}{N}\sum_{i=1}^{N-1}\left|\int_{0}^{\frac{1}{N}} Q_N\left(x,\varepsilon^0\right)dx\right|-O\left(1\right) \notag \\ &=&D_N\left(\varepsilon^0\right)-O\left(1\right).        \notag 
		\end{eqnarray}
		By using theorem \ref{theorem6} (see(\ref{12})) we get that (see (\ref{20}))
		\begin{equation} \label{21}
		\overline{\lim_{N\rightarrow \infty }}\left|\int_{0}^{1}f_N\left(x\right)Q_N\left(x,\varepsilon^0\right)dx\right|=\overline{\lim_{N\rightarrow \infty}}D_N\left(\varepsilon^0\right)=+\infty.
		\end{equation}
		We have (see (\ref{13})) 
		\begin{equation*}
		\parallel f_N \parallel_{Lip1}=\sup_{x,y \in \left[0,1\right]} \left|\frac{f_N\left(x\right)-f_N\left(y\right)}{x-y}\right|+\parallel f_N \left(x\right) \parallel_C =2.
		\end{equation*}
	Then if $f\in Lip1,$ as $\parallel f \parallel_C \leq \parallel f \parallel_{Lip1}$ using the Cauchy inequality we get
	\begin{eqnarray*}
	\left|\int_{0}^{1}f(x)Q_N(x,\varepsilon^0)dx\right| &\leq& \parallel f \parallel_C \int_{0}^{1} \left|Q_N(x,\varepsilon^0)\right|dx \\
	&\leq& \parallel f \parallel_{Lip1} \left(\int_{0}^{1} Q_N ^2(x,\varepsilon^0)dx\right)^{\frac{1}{2}} 	\\
	&\leq& \left(\sum_{k=1}^{N}d_k^2 k\right)^{\frac{1}{2}} \parallel f \parallel_{Lip1}.
	\end{eqnarray*}
	
		Since
		\begin{equation*}
		\int_{0}^{1}f\left(x\right)Q_N\left(x,\varepsilon^0\right)dx
		\end{equation*}
		are linear and bounded functionals on $Lip1$ class and $\parallel f_N \parallel_{Lip1} = 2,$ then according to Banach-Steinhaus theorem and (\ref{21}) we obtain that there exists function $f_0 \in Lip1$ such that 
		
		\begin{equation*}
		\overline{\lim_{N\rightarrow \infty }}\left|\int_{0}^{1}f_0\left(x\right)Q_N\left(x,\varepsilon^0\right)dx\right|=+\infty.
		\end{equation*}
		Hence, by taking into account equality (\ref{8}) we have that 
		\begin{equation*}
		\sum_{k=1}^{\infty}\left|c_k\left(f_0\right)\right|k^{\frac{1}{2}}d_k=+\infty.
		\end{equation*}
		Theorem \ref{theorem6} is proved.
	\end{proof}
	
	\section{Efficiency}

	\begin{theorem} \label{theorem7}
		For $\left(\chi_k\right)$ Haar system (see \cite{KashinSaakyan}, p.57) condition (\ref{5.1}) is fulfilled for 
		\begin{eqnarray}
		d_k=\frac{1}{\log^{1+\varepsilon}\left(k+1\right)} \label{22}
		\end{eqnarray}
		($\varepsilon>0 $ is an arbitrary number).
		\begin{proof}
			Using the definition of Haar's system we get when $k=2^s +l, l<2^s, \chi _k(x)=2^{\frac{s}{2}}$ when $x\in \left(\frac{2l-2}{2^{s+1}},\frac{2l-1}{2^{s+1}}\right),$ $\chi_k(x)=-2^{\frac{s}{2}}$ when $x\in \left(\frac{2l-1}{2^{s+1}},\frac{2l}{2^{s+1}}\right)$ and $\chi_k(x)=0$ when   $x\notin \left[\frac{l-1}{2^{s}},\frac{l}{2^{s}}\right].$ Thus 
			$$\left|\int_{0}^{x}\chi_k(u)du\right|\leq2^{\frac{s}{2}} \text{  when } x\in\left[\frac{l-1}{2^{s}},\frac{l}{2^{s}}\right]  $$ 
			and
			$$\left|\int_{0}^{x}\chi_k(u)du\right|=0 \text{  when } x\notin\left[\frac{l-1}{2^{s}},\frac{l}{2^{s}}\right] $$
			Consequently $\left(i=1,2,\dots, n-1\right)$ we have
			\begin{eqnarray}
			\left|\int_{0}^{\frac{i}{n}}\sum_{k=2^s}^{2^{s+1}-1}k^{\frac{1}{2}}\chi_{k}(x)d_k \varepsilon_k dx\right|\leq \sqrt{2}d_{k\left(s\right)}, \label{23}
			\end{eqnarray}
			where $2^s\leq k\left(s\right)<2^{s+1}.$
			Hence (see (\ref{22}))
			\begin{eqnarray}
			d_{k\left(s\right)}\leq\frac{1}{\log^{1+\varepsilon}\left(2^s\right)}=\frac{1}{s^{1+\varepsilon}}. \label{24}
			\end{eqnarray}
			Next, if $n=2^m +q, q<2^m$ we get (see (\ref{23}) and (\ref{24}))
			
			\begin{eqnarray}
			&&\frac{1}{n}\sum_{i=1}^{n-1}\left|\int_{0}^{\frac{i}{n}}\sum_{k=1}^{n}d_k k^{\frac{1}{2}}\chi_{k}\left(x\right) \varepsilon_k dx\right| \notag \\
			&\leq&\frac{1}{n}\sum_{i=1}^{n-1}\left|\int_{0}^{\frac{i}{n}}\sum_{k=2^m}^{2^m +q}d_k k^{\frac{1}{2}}\chi_{k}\left(x\right) \varepsilon_kdx\right| \notag \\ 
			&+&\frac{1}{n}\sum_{i=1}^{n-1}\left|\sum_{s=0}^{m-1}\int_{0}^{\frac{i}{n}}\sum_{k=2^s}^{2^{s+1}-1}d_k k^{\frac{1}{2}}\chi_{k}\left(x\right) \varepsilon_k dx\right| \notag \\
			&\leq& \frac{\sqrt{2}}{n} \sum_{i=1}^{n-1}d_{k\left(m\right)}\sqrt{k\left(m\right)}2^{\frac{-m}{2}}+\frac{\sqrt{2}}{n} \sum_{i=1}^{n-1} \sum_{s=1}^{m-1}d_{k\left(s\right)}\sqrt{k\left(s\right)}2^{\frac{-s}{2}} \notag \\
			&\leq&\frac{\sqrt{2}}{m^{1+\varepsilon}}+\sqrt{2}\sum_{s=1}^{m-1}\frac{1}{s^{1+\varepsilon}}=O\left(1\right). \notag
			\end{eqnarray}
			Theorem \ref{theorem7} is proved. 
		\end{proof}
	\end{theorem}

\end{document}